\documentclass{amsart}

\usepackage{amsmath, amssymb}
\input xy
\xyoption{all}
\begin{document}

\newtheorem{theorem}{Theorem}[section]
\newtheorem{lemma}[theorem]{Lemma}
\newtheorem{corollary}[theorem]{Corollary}
\newtheorem{fact}[theorem]{Fact}
\newtheorem{proposition}[theorem]{Proposition}
\newtheorem{claim}[theorem]{Claim}
\theoremstyle{definition}
\newtheorem{example}[theorem]{Example}
\newtheorem{remark}[theorem]{Remark}
\newtheorem{definition}[theorem]{Definition}
\newtheorem{question}[theorem]{Question}
\def\id{\operatorname{id}}
\def\PP{\mathbb{P}}
\def\cb{\operatorname{Cb}}
\def\tp{\operatorname{tp}}
\def\stp{\operatorname{stp}}
\def\acl{\operatorname{acl}}
\def\dcl{\operatorname{dcl}}
\def\eq{\operatorname{eq}}
\def\th{\operatorname{Th}}
\def\locus{\operatorname{loc}}

\title{On canonical bases and internality criteria}
\author{Rahim Moosa}
\thanks{Rahim Moosa was supported by an NSERC Discovery Grant.}
\address{Department of Pure Mathematics\\
Faculty of Mathematics\\
University of Waterloo\\
Waterloo\\
Canada N2L 3G1}
\email{rmoosa@math.uwaterloo.ca}
\author{Anand Pillay}
\thanks{Anand Pillay was supported by a Marie Curie Chair.
He also thanks the Humboldt-Stiftung for support during a "Research Awardee
follow-up visit" when some of the work on this paper was done.}
\address{Department of Pure Mathematics\\
School of Mathematics\\
University of Leeds\\
Leeds\\
United Kingdom LS2 9JT}
\email{pillay@maths.leeds.ac.uk}

\date{May 24, 2007}

\begin{abstract}
A criterion is given for a strong type in a finite rank stable theory $T$ to be (almost) internal to a given nonmodular minimal type.
The motivation comes from results of Campana~\cite{campana81} which give criteria for a compact complex analytic space to be ``algebraic" (namely
Moishezon). The {\em canonical base property} for a stable theory states that the type of the canonical base of a stationary type over a realisation is almost internal to the minimal types of the theory.
It is conjectured that every finite rank stable theory has the canonical base property.
It is shown here, that in a theory with the canonical base property, if $p$ is  a stationary type for which there exists a family of types $q_b$, each internal to a non-locally modular minimal type $r$, and such that any pair of independent realisations of $p$ are ``connected'' by the $q_b$'s, then $p$ is almost internal to $r$.
\end{abstract}

\maketitle

\section{Introduction}
This paper is concerned with analysability and internality in the context of
stable theories of finite rank.
While we will briefly recall these notions at the end of this introduction, we refer the reader to~\cite{pillay96} for details on geometric stability theory.

In the many sorted 
structure $\mathcal{A}$ of compact complex spaces, a space $X$ will be an algebraic variety (or rather a 
Moishezon 
space) iff its generic type is internal to the sort of the projective line.
Roughly speaking, Campana proves in~\cite{campana81} that if the 
space $X$ is of K\"ahler-type and ``algebraically connected" in the sense
that there exists an analytic family of {\em algebraic} subvarieties of
$X$ such that any two points of the space $X$ are connected by a finite sequence of
algebraic subvarieties from this family, then $X$ is itself algebraic.

We wanted to find a general
model-theoretic treatment or at least analogue of this result, and this is what the current paper is about.

We work in a saturated model $\overline{M}^{\eq}$ of a complete stable theory $T$ of finite $U$-rank
sort-by-sort. In general we will be concerned with types (over arbitrary small sets of parameters) of possibly infinite tuples that are contained in the algebraic closure of a finite tuple.

For $p(x)\in S(A)$ a stationary type, the
{\em canonical base of $p$} usually means (an enumeration of) the smallest definably closed subset $A_0$ of
$\dcl(A)$ such that $p(x)$ does not fork over $A_0$ and the restriction
of $p$ to $A_0$ is stationary.
It will be more convenient for us to consider (an enumeration of) the {\em algebraic} closure of $A_0$, which we will denote by $\cb(p)$ and also refer to as the canonical base.
This abuse of notation is harmless since throughout this paper we will only be concerned with precision up to interalgebraicity.
We will also write $\cb(a/A)$ to mean $\cb\big(\stp(a/A)\big)$.

The following is a property that is conjectured to hold in all finite rank theories.
Let $\mathbb{P}$ be the set of all nonmodular minimal (stationary $U$-rank $1$) types.

\medskip
\noindent
{\bf Canonical Base Property (CBP).}
{\em If $b=\cb(a/b)$ then $\stp(b/a)$
is almost $\mathbb{P}$-internal.}
\medskip

\begin{remark}
\begin{itemize}
\item[(a)]
The CBP is preserved by naming parameters.
\item[(b)]
The CBP is equivalent to the statement where one replaces $\mathbb{P}$ by the set of {\em all} minimal types. The reason is that from Proposition 1.9 of~\cite{chatzidakis06} (see also ~\cite{pillay-cb}) we know that $\stp(b/a)$ is already (almost) analysable in $\mathbb{P}$, hence orthogonal to all modular minimal types.
\end{itemize}
\end{remark}

The following definition is inspired by a related notion for cycle spaces of compact complex manifolds introduced by Campana in~\cite{campana81}.

\begin{definition}
\label{ample}
Suppose $q(x,y),s(y),p(x)\in S(A)$ are stationary types.
We say that $q$ is a {\em generating family for $p$ over $s$} if
\begin{itemize}
\item[(i)]
$q(x,y)\vdash p(x)\wedge s(y)$,
\item[(ii)]
$q_b(x):=q(x,b)$ is stationary for $b\models s$, and
\item[(iii)]
if $(a,b)\models q$ then $\acl(Aa)\cap\acl(Ab)=\acl(A)$
\end{itemize}
We call $s$ the {\em parameter space} for the family, and  $q_b$ the {\em fibres}, where $b\models s$.
If $s$ is algebraic then we say the family is {\em trivial}.
A generating family is {\em canonical} if
$b\in\cb(q_b)$ for $b\models s$.
It is said to {\em almost separate points of $p$}  if for any $a\models p$ there are only finitely many other realisations of $p$ that lie on all the same fibres as $a$.
\end{definition}

Note that to be given a generating family for a stationary type $\tp(a/A)$ is to be given a tuple $b$ such that $\acl(Aa)\cap\acl(Ab)=\acl(A)$, $\tp(b/A)$ is stationary, and $\tp(a/Ab)$ is stationary.
The generating family is then canonical if $b\in\cb(a/Ab)$, and trivial if $\tp(b/A)$ is algebraic.

\begin{remark}
\label{genrem}
\begin{itemize}
\item[(a)]
If $q$ almost separates points of $p$ then $p^{\overline{M}}\subset\acl(As^{\overline{M}})$.
\item[(b)]
If $q$ is a nontrivial canonical generating family with fibres of $U$-rank $1$, then $q$ almost separates points.
\item[(c)]
If $p$ has a nontrivial canonical generating family then $p$ is not $1$-based.
\end{itemize}
\end{remark}

\begin{proof}
Part (a) is clear.

For part (b) note that as $b\notin\acl(A)$, condition~(iii) of the definition implies that $b\notin \acl(Aa)$.
Hence there exists $b'\notin \acl(Aab)$ realising $\tp(b/Aa)$.
In particular $b'\notin \acl(Ab)$.
As $b'\in\cb(q_{b'})$, we must have 
that $q_{b}\cup q_{b'}$ is not a nonforking extension of both $q_{b}$ and $q_{b'}$, which implies by the rank hypothesis that $q_{b}\cup q_{b'}$ is algebraic.
So $q$ almost separates points of $p$.

For part (c) note that if $p$ is $1$-based then $\cb(q_b)\in\acl(Aa)$, and so by condition~(iii), $b\in\acl(A)$.
That is, the family is trivial.
\end{proof}

The following lemma justifies the term ``generating family"; it says that if $q$ is a generating family for 
$p$ over $s$, and given independent realisations $a$ and $a'$ of $p$, one can get from $a$ to $a'$ by 
``moving along the fibres" of $q$ over $s$. The lemma is due originally to Lascar ~\cite{Lascar} who
stated it in the language of groups of automorphisms, but we give a proof for the sake of completeness.

\begin{lemma}
\label{justify}
The following are equivalent:
\begin{itemize}
\item[(i)]
$\acl(Aa)\cap\acl(Ab)=\acl(A)$
\item[(ii)]
For any $a'\models\stp(a/A)$ independent of $a$ over $A$, there exist $a=a_0,\dots,a_\ell=a'$ and $b=b_0,\dots,b_\ell$ such that
\begin{itemize}
\item
$a_{i+1}\models\stp(a_i/Ab_i)$ and
\item
$b_{i+1}\models\stp(b_i/Aa_{i+1})$,
\end{itemize}
for all $i=0,\dots,\ell-1$.
\end{itemize}
\end{lemma}
\begin{proof}
We suppress the parameters $A$ by naming them as constants to the language.
Assume (i).
By stationarity, in order to show (ii) it suffices to find {\em some} $a'\models\stp(a)$ independent of $a$ satisfying the desired properties.
This is what we do.
Define sequences $a=a_0,a_1,\dots$ and $b=b_0,b_1,\dots$ inductively so that
\begin{itemize}
\item[(1)]
$a_{i+1}$ realises the nonforking extension of $\stp(a_i/b_i)$ to $ab_i$
\item[(2)]
$b_{i+1}$ realises the nonforking extension of $\stp(b_i/a_{i+1})$ to $aa_{i+1}$
\end{itemize}
Note that $U(a_{i+1},b_i/a)\geq U(a_i,b_i/a).$
Indeed,
\begin{eqnarray*}
U(a_{i+1},b_i/a)
& = &
U(a_{i+1}/ab_i)+U(b_i/a)\\
 & = &
U(a_{i+1}/b_i)+U(b_i/a)\\
& = &
U(a_i/b_i)+U(b_i/a)\\
& \geq &
U(a_i/ab_i)+U(b_i/a)\\
& = &
U(a_i,b_i/a)
\end{eqnarray*}
Similarly, $U(a_{i+1},b_{i+1}/a)\geq U(a_{i+1},b_i/a)$.
So we get a sequence $$U(a,b/a)\leq U(a_1,b/a)\leq U(a_1,b_1/a)\leq U(a_2,b_1/a)\leq\cdots\leq U(a,b).$$
Hence, for some $\ell\geq 0$ it must be the case that $U(a_\ell,b_\ell/a)=U(a_{\ell+1},b_\ell/a)$.
It follows that
\begin{eqnarray*}
U(a_\ell/ab_\ell)+U(b_\ell/a)
& = &
U(a_\ell,b_\ell/a)\\
& = &
U(a_{\ell+1},b_\ell/a)\\
& = &
U(a_{\ell+1}/ab_\ell)+U(b_\ell/a)\\
& = &
U(a_{\ell+1}/b_\ell)+U(b_\ell/a)\\
& = &
U(a_\ell/b_\ell)+U(b_\ell/a)
\end{eqnarray*}
That is, $U(a_\ell/ab_\ell)=U(a_\ell/b_\ell)$ and so $a_\ell$ and $a$ are independent over $b_\ell$.
On the other hand, $b_\ell$ and $a$ are independent over $a_\ell$ by construction.
So
$\cb(a/a_\ell b_\ell)\subseteq\acl(a_\ell)\cap\acl(b_\ell)$.
But $a_\ell b_\ell$ has the same type as $ab$, and $\acl(a)\cap\acl(b)=\acl(\emptyset)$.
What we have shown is the following statement that we will use again later:
\begin{itemize}
\item[($*$)]
If $\acl(a)\cap\acl(b)=\acl(\emptyset)$ and $a=a_0,a_1,\dots$ and $b=b_0,b_1,\dots$ satisfy~(1) and~(2),  then there exists $\ell\geq 0$ such that $a$ is independent of $a_\ell b_\ell$ over $\emptyset$.
\end{itemize}
In particular, $a_{\ell}$ and $a$ are independent over the empty set.
Setting $a'=a_\ell$, we have shown (ii).

For the converse, suppose (ii) holds and let $\sigma_0,\dots,\sigma_{\ell-1}$ and $\tau_1,\dots,\tau_\ell$ be automorphisms such that $\sigma_i(a_i b_i)=a_{i+1}b_i$ and $\tau_{i+1}(a_{i+1}b_i)=a_{i+1}b_{i+1}$, for all $i=0,\dots,\ell-1$.
Hence, for each $i$, $\sigma_i$ witnesses that
$\acl(a_i)\cap\acl(b_i)=\acl(a_{i+1})\cap\acl(b_i)$
and $\tau_{i+1}$ witnesses that
$\acl(a_{i+1})\cap\acl(b_i)=\acl(a_{i+1})\cap\acl(b_{i+1})$.
So $\acl(a)\cap\acl(b)=\acl(a')\cap\acl(b_\ell)$.
In particular, $\acl(a)\cap\acl(b)\subseteq\acl(a)\cap\acl(a')$.
But the independence of $a$ and $a'$  implies that the latter is $\acl(\emptyset)$, as desired.
\end{proof}

The purpose of this note is to prove the following theorem which is motivated by, and in part recovers, Campana's ``algebraicity criteria'' for compact K\"ahler manifolds (cf. Th\'eor\`emes ~2 and~3 of~\cite{campana81}). It depends heavily on the results of Chatzidakis ~\cite{chatzidakis06}.

\begin{theorem}
\label{ample-by-moishezon=moishezon}
Suppose that a stationary type $p(x)\in S(A)$ has a generating family with a fibre that is almost internal to a nonmodular minimal type $r$.
\begin{itemize}
\item[(i)]
Then $r$ is nonorthogonal to $A$ (and so {\em every} fibre of the generating family is almost internal to $r$) and $p$ is almost analysable in $r$.
\item[(ii)]
 If moreover the CBP holds for $T$, then $p$ is almost internal to $r$.
\end{itemize}
\end{theorem}

The rest of the paper is organised as follows.
In Section~\ref{ccm} we elaborate on the connections between the notions presented here and the model theory of compact complex manifolds, in Section~\ref{cbpconsequences} we discuss a natural strengthening of the CBP which is true in compact complex manifolds and relate it to the work of Zo\'e Chatzidakis~\cite{chatzidakis06}, and finally in Section~\ref{criteria} we prove Theorem~\ref{ample-by-moishezon=moishezon}.

\medskip

We end this introduction with a brief discussion of internality and analysability. If $p(x)\in S(A)$ is stationary and $\bf P$ is some
family of partial types (over possibly differing sets) we say $p$ is {\em 
internal} to $\bf P$ or is $\bf P$
-internal, if there is
some $C$ containing $A$ and $a$ realizing $p$ independent from $C$ over $A$, such that 
$a$ is in the {\em definable closure} of $C$ together with some realisations of some partial types in $\bf P$ 
whose domains are contained in $C$. If we replace 
definable closure by {\em algebraic closure} we get the notion of $p$ being {\em almost internal} to $\bf P$. 
We will be using this notion in at least two cases, first where $\bf P$ is some $\acl(\emptyset)$-invariant
family of minimal (stationary $U$-rank $1$) types over varying domains, such as the family of all nonmodular 
minimal types, and second where $\bf P$ is a single 
minimal type (over some set) $q$ say. In the second case the hypothesis that $p\in S(A)$ is almost internal
to $q$ implies that $q$ is nonorthogonal to all of its conjugates over $\acl(A)$ -- that is, $q$ is {\em nonorthogonal to $A$}.
So almost internality of $p$ to $q$ is equivalent to almost internality of $p$ to $\bf Q$ where 
$\bf Q$ is the family of conjugates of $q$ over $\acl(A)$. 

We also have the standard notion of analysability; $p(x)\in S(A)$ is
{\em analysable} in $\bf P$ if there are $a_0\dots a_n$ such that 
$\stp(a_{i+1}/Aa_{0}\dots a_{i})$ is internal
to $\bf P$ for $i=0,\dots,n-1$ and $a\in \acl(a_{n})$. So algebraic closure is built into the definition. We 
might also want to define {\em almost analysability} by only requiring that $\stp(a_{i+1}/Aa_0\dots a_{i})$ be 
almost internal to $\bf P$ for $i=0,\dots,n-1$, and in this case we may as well require that $a_n = a$. Assuming 
that the family $\bf P$ is $\acl(A)$-invariant, it is not hard to see that $p(x)\in S(A)$ is almost analysable
in $\bf P$ iff $p(x)$ is analysable in $\bf P$. 
As in the case of internality, if $q$ is a minimal type nonorthogonal to $A$ then almost analysability of $p$ in $q$ is equivalent to (almost) analysability of $p$ in $\bf Q$ where 
$\bf Q$ is the family of conjugates of $q$ over $\acl(A)$.

\bigskip

\section{Compact complex manifolds}
\label{ccm}
Much of the discussion in the introduction -- including the CBP, the notion of a generating family, as well as the statement of Theorem~\ref{ample-by-moishezon=moishezon} itself -- is informed by certain aspects of the model theory of compact complex manifolds.
In this section we aim to make these origins and connections precise.
The material here is almost entirely of an expository nature.
Let $\mathcal{A}$ denote the multi-sorted structure where there is a sort for each irreducible compact complex-analytic space and the language consists of a predicate for each complex-analytic subset of each finite cartesian product of sorts.
The theory $\th(\mathcal{A})$ admits quantifier elimination and is, sort-by-sort, of finite Morley rank.
A survey of some of the model theory of $\th(\mathcal{A})$ can be found in~\cite{moosa-ccs}.

We work in a fixed sufficiently saturated elementary extension $\mathcal{A}'$ of $\mathcal{A}$.
Here it is more natural to work only with types of finite tuples.
Since the theory is totally transcendental there will be no loss of generality in doing so.
It may however lead to some (harmless) abuse of notation; for example, we will often write $\cb(p)$ when we really mean a finite tuple whose algebraic closure is $\cb(p)$.

Among the sorts of $\mathcal{A}$ we have the projective line $\mathbb{P}(\mathbb{C})$ in which the complex field is definable.
In $\mathcal{A'}$ the interpretation of this sort is $\mathbb{P}(\mathbb{C}')$ where $\mathbb{C}'$ is the corresponding elementary extension of the complex field.

The nonmodular minimal types in this theory are exactly those that are nonorthogonal to the generic type of the projective line (which is equal to the generic type of $\mathbb{C}'$).
Indeed, this follows from the truth of the Zilber dichotomy in $\th(\mathcal{A})$ via Zariski geometries plus the fact that the only infinite field definable in $\mathcal{A'}$ is $\mathbb{C}'$ (see Corollary~4.8 of~\cite{ret}).
However, it also follows directly from a theorem of Campana~\cite{campana80} (due independently to Fujiki~\cite{fujiki82}) as observed by the second author in~\cite{pillay01}.
We will discuss this theorem of Campana's later, as it is tied up with the issues we are concerned with in this paper.
In any case, it follows that almost internality to the set of nonmodular minimal types coincides in this theory with almost internality to the generic type of the projective line.
Moreover, as the following summarising fact describes, almost internality to the projective line has a very natural geometric meaning.

First some notation: 
given a tuple $a$ from $\mathcal{A'}$ by the {\em locus} of $a$, denoted $\locus(a)$, we mean the smallest complex-analytic set whose interpretation in $\mathcal{A'}$ contains $a$.

\begin{fact}
\label{algebraicity}
Suppose $p(x)=\tp(a/b)$ is a stationary type.
Let $X=\locus(a)$, $S=\locus(b)$, $G=\locus(ab)\subseteq X\times S$, and $G\to S$ and $G\to X$ the natural projections.
Then the following are equivalent:
\begin{itemize}
\item[(a)]
$p$ is almost internal to the set of nonmodular minimal types.
\item[(b)]
$p$ is almost internal to the projective line.
\item[(c)]
$p$ is internal to the projective line.
\item[(d)]
After base change $G\to S$ is a {\em Moishezon morphism}: that is, for some complex-analytic space $T$ over $S$, the fibred-product $G\times_ST$ bimeromorphically embeds into a projective linear space $\mathbb{P}(\mathcal{F})$ over $T$, where $\mathcal{F}$ is some coherent analytic sheaf on $T$.
\item[(e)]
For some complex-analytic space $\widehat{T}$ over $S$, the fibred-product $G\times_S\widehat{T}$ bimeromorphically embeds into $\mathbb{P}_n(\mathbb{C})\times \widehat{T}$ over $\widehat{T}$, for some $n\geq 0$.
\end{itemize}
\end{fact}
\begin{proof}
The equivalence of (c), (d), and (e) is Proposition~4.4 of~\cite{ret}.
In the preceeding discussion we explained the equivalence of (a) and (b).
It remains to prove that (b) implies (c).
In the case when the tuple $b$ comes from the
standard model $\mathcal{A}$ this is stated in ~\cite{pillay2000} and boils down to the fact that if $X$ is a Moishezon space and $X'$ is another space such that there is dominant 
generically finite-to-one meromorphic map from $X'$ to $X$, then $X'$
is also Moishezon.
To prove that (b) implies (c) in general we actually need a relative version of the above fact, which we now explain.
If $\tp(a/b)$ is almost internal to
the projective line then there exist a tuple $c$ extending $b$ such that $a$ is independent of $c$ over $b$, and a tuple $d$ from $\mathbb{P}(\mathbb{C})$ such that $a\in\acl(cd)$.
It suffices to prove that $\tp(a/c)$ is $\mathbb{P}(\mathbb{C})$-internal.
Let $T=\locus(c)$, $H=\locus(ac)\subseteq X\times T$, $P=\locus(dc)\subseteq \mathbb{P}^n(\mathbb{C})\times T$, and $Z=\locus(adc)\subseteq H\times_T P$.
We then have the following diagram where all the maps are the natural projections:
$$
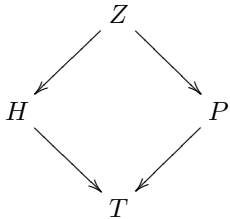
\xymatrix{& Z\ar[dl]\ar[dr] &\\H\ar[dr] & &P\ar[dl]\\& T}$$
Now $P\to T$ is Moishezon as $P\subseteq\mathbb{P}^n(\mathbb{C})\times T$.
Also $Z\to P$ is generically finite-to-one since $a\in\acl(cd)$, and all generically finite-to-one maps are Moishezon.
Hence the composition $Z\to T$ is Moishezon.
(See for example Section~1 of~\cite{fujiki82} for these facts about Moishezon morphisms.)
By the equivalence of parts (c) and (d) this implies that $\stp(ad/c)$ is $\mathbb{P}(\mathbb{C})$-internal.
Hence $\tp(a/c)$ is $\mathbb{P}(\mathbb{C})$-internal.
\end{proof}

In particular, note that if $p$ is a stationary type over $\emptyset$, then $p$ is (almost) internal to the 
projective line if and only if it is the generic type of a {\em Moishezon variety}: an irreducible 
complex-analytic space that is bimeromorphic to a projective variety.

We now explain why the CBP is true in $\th(\mathcal{A})$.
This is a consequence of a theorem of Campana~\cite{campana80} already referred to earlier.
Indeed, it was this theorem in complex analytic geometry that inspired the second author to introduce the CBP in~\cite{pillay01}.
To describe Campana's theorem we need to recall the Barlet space of cycles.
For $X$ any complex-analytic space, a {\em $k$-cycle of $M$} is a finite linear combination $\displaystyle Z=\sum_in_iZ_i$ where the $Z_i$'s are distinct $k$-dimensional irreducible compact complex-analytic subsets of $X$, and each $n_i$ is a positive integer called the {\em multiplicity} of $Z_i$ in $Z$.
By $|Z|$ we mean the underlying set or {\em support} of $Z$, namely $\displaystyle \bigcup_iZ_i$.
We denote the set of all $k$-cycles of $X$ by $\mathcal{B}_k(X)$, and the set of all cycles of $X$ by $\displaystyle \mathcal{B}(X):=\bigcup_k\mathcal{B}_k(X)$.
In~\cite{barlet} Barlet endowed $\mathcal{B}_k(X)$ with a natural structure of a complex-analytic space whereby if for $s\in \mathcal{B}_k(X)$ we let $Z_s$ denote the cycle respresented by $s$, then the set
$\{(x,s):s\in\mathcal{B}_k(X), x\in |Z_s|\}$
is a complex-analytic subset of $X\times \mathcal{B}_k(X)$.
Equipped with this complex structure, $\mathcal{B}(X)$ is called the {\em Barlet space of $X$}.
When $X$ is a projective variety the Barlet space coincides with the Chow scheme.
In~\cite{campanathesis} it is shown that
$$\mathcal{B}^*(X):=\{s\in\mathcal{B}(X):Z_s\text{ is irreducible with multiplicity }1\}$$
 is a {\em Zariski open} subset of  $\mathcal{B}(X)$: its complement in $\mathcal{B}(X)$ is a proper complex-analytic subset.
An irreducible component of $\mathcal{B}(X)$ is {\em prime} if it has nonempty intersection with $\mathcal{B}^*(X)$.

Note that even for a compact complex-analytic space $X$ it is not necessarily the case that the (prime) irreducible components of $\mathcal{B}(X)$ are again compact.
Indeed, the condition that all the prime components of $\mathcal{B}(X^n)$ for all $n\geq 0$ turns out to be important model-theoretically; it is equivalent to the property introduced by the first author in~\cite{sat} of being {\em essentially saturated}.
This property is satisfied for example, by all holomorphic images of compact K\"ahler manfiolds (these are the so-called {\em K\"ahler-type} spaces introduced by Fujiki in~\cite{fujiki78}).

\begin{theorem}[Campana~\cite{campana80}]
\label{campanacbp}
Suppose $X$ is a compact complex-analytic space and $S$ is an irreducible compact complex-analytic subset of $\mathcal{B}(X)$ such that $S\cap\mathcal{B}^*(X)\neq\emptyset$.
Let $Z_S:=\{(x,s):s\in S, x\in |Z_s|\}\subseteq X\times S$ denote the {\em graph} of the family of cycles $S$.
Then the natural projection $p_X:Z_S\to X$ is a Moishezon morphism.\qed
\end{theorem}

Let us see how this implies the CBP.
Given a stationary type $p(x)=\tp(a/b)$ let $X=\locus(a), Y=\locus(b)$, and $G=\locus(a,b)\subseteq X\times Y$.
By stationarity the general fibres of $G$ over $Y$ are irreducible complex-analytic subsets of $X$.
The universal property of the Barlet space, along with the geometric flattenning theorem of~\cite{barlet79}, gives us a meromorphic map $\phi:Y\to\mathcal{B}(X)$ such that for general $y\in Y$, $G_y=Z_{\phi(y)}$ (see for example Proposition~2.20 of~\cite{campanapeternell94}).
The condition that $b=\cb(p)$ translates into the statement that $\phi$ is generically finite-to-one onto its image.
Let $S$ denote the image of $Y$ under $\phi$.
Then $S$ is an irreducible compact complex-analytic subset, $S\cap\mathcal{B}^*(X)\neq\emptyset$, and $\id\times\phi$ induces a dominant generically finite-to-one meromorphic map from $G$ to $Z_S$ over $X$.
By Theorem~\ref{campanacbp}, $p_X:Z_S\to X$, and hence $G\to X$, is Moishezon.
By Fact~\ref{algebraicity}, $\stp(b/a)$ is internal to the projective line.
That is, the CBP holds in $\th(\mathcal{A})$.

Finally in this section we explain the origins of our notion of a generating family for a stationary type.
Fix an irreducible compact complex-analytic space $X$.
In~\cite{campana81} Campana calls a family of cycles $(Z_s:s\in S)$, where $S$ is an irreducible compact complex-analytic subset of $\mathcal{B}(X)$ with $S\cap\mathcal{B}^*(X)\neq\emptyset$, a {\em generating} family for $X$ if for any $x,x'\in X$ there exists a sequence $x=x_0,x_1,\dots,x_\ell=x'$ in $X$ and $s_1,\dots,s_\ell\in S$ such that $x_i$ and $x_{i+1}$ both lie in the cycle $|Z_{s_{i+1}}|$, for all $i=0,\dots,\ell-1$.
That is, if every pair of points in $X$ can be connected by moving along the cycles in the family.
The motivation for Definition~\ref{ample} is made explicit by the following:

\begin{proposition}
\label{mtgen-campanagen}
Suppose $q=\tp(a,b)$ is a canonical generating family for $p=\tp(a)$ over $s=\tp(b)$ in the sense of Defintion~\ref{ample}.
Let $X=\locus(a), Y=\locus(b)$, and $G=\locus(a,b)\subseteq X\times Y$.
Let $\phi:Y\to\mathcal{B}(X)$ be the meromorphic map given by the universal property of the Barlet space, as discussed above.
Let $S$ denote the image of $Y$ under $\phi$.
Then $(Z_s:s\in S)$ is a generating family of cycles for $X$.
\end{proposition}

\begin{proof}
Note that $\phi$ is generically finite-to-one onto $S$ and that for general $y\in Y$, $G_y=Z_{\phi(y)}$.
As in section~1 of~\cite{campana81}, for each $n\in\mathbb{N}$, let $R^n\subseteq X\times X$ be the set of pairs $(x,x')$ for which there exist $x=x_0,\dots,x_n=x'\in X$ and $s_1,\dots,s_n\in S$ such that 
$x,x_1\in |Z_{s_1}|$, $x_1,x_2\in |Z_{s_2}|$, $\dots$, $x_{n-1},x'\in |Z_{s_n}|$.
By construction $R^n$ is a complex-analytic subset of $X\times X$.
To show that $(Z_s:s\in S)$ is a generating family of cycles for $X$ we need to show that $R^\ell=X\times X$ for some $\ell\in\mathbb{N}$.
Now, since $q$ is a generating family for $p$, Lemma~\ref{justify} implies there exists $a'\models p$ independent from $a$ and $\ell\in\mathbb{N}$ for which there exist $a=a_0,\dots,a_\ell=a'$ and $b=b_0,\dots,b_\ell$ such that
$a_{i+1}\models\stp(a_i/Ab_i)$ and
$b_{i+1}\models\stp(b_i/Aa_{i+1})$,
for all $i=0,\dots,\ell-1$.
In particular the $b_i$'s are all generic in $Y$ and hence $\phi$ is defined on them, and $a_i,a_{i+1}\in G_{b_i}=Z_{\phi(b_i)}$.
That is, $(a,a')\in R^\ell$.
Since $a$ and $a'$ are generic independents of $X$, and $X$ is irreducible, it follows that $R^\ell=X\times X$.
\end{proof}
Campana proves the following algebraicity criterion (cf. Th\'eor\`eme 3 of~\cite{campana81}):
\begin{theorem}[Campana~\cite{campana81}]
\label{campanacriterion}
Suppose $X$ is a K\"ahler-type complex-analytic space.
If $(Z_s:s\in S)$ is a generating family of cycles for $X$, and each $Z_s$ is Moishezon, then $X$ is Moishezon.\qed
\end{theorem}
Using Proposition~\ref{mtgen-campanagen} it is not hard to see that our main result (Theorem~\ref{ample-by-moishezon=moishezon}) specialised to compact complex manifolds is a consequence of the above algebraicity criterion of Campana.
In fact, Theorem~\ref{ample-by-moishezon=moishezon} can be viewed as generalising the ``generic'' content of Theorem~\ref{campanacriterion} to all finite rank theories with the CBP.

\bigskip
\section{Variants of the CBP}
\label{cbpconsequences}

Zo\'e Chatzidakis~\cite{chatzidakis06} has shown that the CBP implies what on the face of it appears to be a 
stronger property.
We will use this in our proof of Theorem~\ref{ample-by-moishezon=moishezon} in Section~\ref{criteria} below.

\begin{theorem}[Chatzidakis~\cite{chatzidakis06}]
\label{strongcbp}
Suppose the CBP holds for $T$.
If $b=\cb(a/b)$ then $\stp\big(b/\acl(a)\cap\acl(b)\big)$ is almost $\PP$-internal.\footnote{Recall that in this general setting $\mathbb{P}$ denotes the collection of all nonmodular minimal types.}\qed
\end{theorem}

The point here is that we are able to conclude that $\stp\big(b/\acl(a)\cap\acl(b)\big)$ rather than just $\stp(b/a)$ is almost $\mathbb{P}$-internal.
The following immediate corollary is our main use of this theorem.

\begin{corollary}
\label{corstrongcbp}
Suppose the CBP holds for $T$ and $q(x,y)$ is a canonical generating family for $p(x)$ over $s(y)$.
Then $s(y)$ is almost $\mathbb{P}$-internal.
\end{corollary}

\begin{proof}
Writing $q(x,y)=\tp(a,b/A)$ we have that $\acl(Aa)\cap\acl(Ab)=\acl(A)$ and $b\in\cb(a/Ab)=:b'$.
Then $\acl(a)\cap\acl(b')\subseteq\acl(A)$ and $b'=\cb(a/b')$.
Hence by Theorem~\ref{strongcbp}, $\stp(b'/A)$ is almost $\mathbb{P}$-internal.
So $\tp(b/A)$ is almost $\mathbb{P}$-internal.
\end{proof}

The proof of Theorem~\ref{strongcbp} involves some substantial model-theoretic technicalities.
It turns out however that in the case of compact complex manifolds the conclusion follows from a very transparent geometric argument.
In this section, generalising from the case of compact complex manifolds, we introduce a ``uniform'' version of the CBP and give a rather conceptual argument for why the conclusion of Theorem~\ref{strongcbp} follows from this strengthened form of the CBP.

To motivate this ``uniform'' CBP, let us look back to the previous section at how the CBP was established for $\th(\mathcal{A})$, and notice that one actually gets more.
Suppose $\tp(a/b)$ is a stationary type, $X=\locus(a)$, $Y=\locus(b)$, and $G=\locus(a,b)\subseteq X\times Y$ and $\pi_X:G\to X$ is the natural projection.
Campana's theorem (Theorem~\ref{campanacbp}) implies that if $b=\cb(p)$ then $\pi_X$ is a Moishezon  morphism, {\em without taking any base change}.
On the other hand, the CBP only concludes that $\tp(b/a)$ is internal to the projective line, which means that $\pi_X$ is Moishezon after a suitable base change (cf. the equivalence of (c) and (d) in Fact~\ref{algebraicity}).
To see concretely how these conclusions differ, consider what happens if there is another Moishezon morphism $f:X\to X'$.
The Moishezonness of $\pi_X:G\to X$ would imply that $f\circ\pi_X$ is Moishezon, and so $\stp\big(b/f(a)\big)$ would be $\mathbb{P}(\mathbb{C})$-internal.
Whereas internality of $\stp(b/a)$ to the projective line only implies that $\stp\big(b/f(a)\big)$ is $\mathbb{P}(\mathbb{C})$-{\em analysable} (in two steps).
So Campana's theorem tells us more than the CBP for $\th(\mathcal{A})$.

This motivates the following definition; we return to our general set-up where $T$ is an arbitrary complete stable theory of finite rank and $\mathbb{P}$ is the collection of all nonmodular  minimal types.

\medskip
\noindent
{\bf Uniform Canonical Base Property (UCBP).}
{\em Suppose $b=\cb(a/b)$ and $C$ is a set of parameters such that $\stp(a/C)$ is almost $\mathbb{P}$-internal.
Then for $b'$ realising the nonforking extension of $\stp(b/a)$ to $\acl(Ca)$, $\stp(ab'/C)$ is almost $\mathbb{P}$-internal.}
\medskip

Note that UCBP implies CBP by taking $C=a$.

\begin{proposition}
$\th(\mathcal{A})$ has the UCBP.
\end{proposition}

\begin{proof}
This is more or less the preceding discussion, but we give some details now.
Suppose $p=\tp(a/b)$ is stationary in $\th(\mathcal{A})$ with $b=\cb(p)$.
We have seen that the projection $\locus(a,b)\to\locus(a)$ is a Moishezon morphism.
Now suppose $C$ is a set of parameters such that $\stp(a/C)$ is $\mathbb{P}(\mathbb{C})$-internal and let $b'$ realise the nonforking extension of $\stp(b/a)$ to $\acl(Ca)$.
Set $c=\cb(a/C)$, $X=\locus(a,b',c)$, $Y=\locus(a,c)$ and $Z=\locus(c)$.
So $X\to Y$ is a Moishezon morphism as it is obtained from $\locus(a,b)\to\locus(a)$ by base change (this is where the choice of $b'$ independent of $c$ over $a$ is used).
We need to show that, possibly after further base change, $X\to Z$ is Moishezon.
Since $\stp(a/c)$ is $\mathbb{P}(\mathbb{C})$-internal, we have by Fact~\ref{algebraicity} that for some $T\to Z$, $Y\times_ZT\to T$ is Moishezon.
On the other hand, $X\times_ZT\to Y\times_ZT$ is Moishezon since it is obtained from $X\to Y$ by base change.
Hence the composition $X\times_ZT\to T$ is Moishezon, as desired.
\end{proof}

The following argument yields the conclusion of Chatzidakis' theorem under the stronger hypothesis of UCBP.

\begin{proposition}
\label{ucbp-zoe}
Suppose the UCBP holds for $T$.
If $b=\cb(a/b)$ then $\stp\big(b/\acl(a)\cap\acl(b)\big)$ is almost $\PP$-internal.
\end{proposition}

\begin{proof}
Working over $\acl(a)\cap\acl(b)$ we may assume that $\acl(a)\cap\acl(b)=\acl(\emptyset)$, and aim to show that $\stp(b)$ is almost $\mathbb{P}$-internal.

Next, we may assume that $a=\cb(b/a)$ also.
Indeed, letting $a'=\cb(b/a)$ and $b'=\cb(a'/b)$ we see that $b$ and $a$ are independent over $b'$.
This implies that $\acl(b)=\acl(b')$.
Since $a'=\cb(b'/a')$, $b'=\cb(a'/b')$, and $\acl(a')\cap\acl(b')=\acl(\emptyset)$, we have obtained the desired reduction.

Now define inductively sequences $a=a_0,a_1,\dots$ and $b=b_0,b_1,\dots $ satisfying the following
\begin{itemize}
\item
$a_{i+1}$ realises the nonforking extension of $\tp(a_i/b_i)$ to $ab_i$
\item
$b_{i+1}$ realises the nonforking extension of $\stp(b_i/a_{i+1})$ to $aa_{i+1}$
\item
$\stp(a_ib_i/a)$ is almost $\mathbb{P}$-internal.
\end{itemize}
This is done as follows.
Applying UCBP to $(a,b,C=a)$ we get that $\stp(b/a)$ is almost $\mathbb{P}$-internal.
Now, appying UCBP to $(b,a,C=a)$ we obtain a realisation $a_1$ of the nonforking extension of $\tp(a/b)$ to $ab$ such that $\stp(a_1b/a)$ is almost $\mathbb{P}$-internal.
Hence $\stp(a_1/a)$ is almost $\mathbb{P}$-internal.
Now apply UCBP to $(a_1,b,C=a)$ to obtain a realisation $b_1$ of the nonforking extension of $\stp(b/a_1)$ to $aa_1$ such that $\stp(a_1b_1/a)$ is almost $\PP$-internal.
Hence, $\stp(b_1/a)$ is almost $\PP$-internal, and we may continue.

By ($*$) in the proof of Lemma~\ref{justify} we know that eventually, for some $\ell$, $a$ is independent of $a_\ell b_\ell$ over $\emptyset$.
In particular, $\stp(b_\ell/a)$ -- which is almost $\mathbb{P}$-internal since $\stp(a_\ell b_\ell /a)$ is almost $\mathbb{P}$-internal -- is a nonforking extension of $\stp(b_\ell)$.
Hence $\stp(b_\ell)$ is almost $\PP$-internal, as desired.
\end{proof}

\begin{question}
Does the CBP imply the UCBP?
Equivalently, does the conclusion of Chatzidakis' Theorem~\ref{strongcbp}
imply the UCBP?
\end{question}

\bigskip

\section{Internality criteria}
\label{criteria}

We now work toward proving Theorem~\ref{ample-by-moishezon=moishezon}.
We will need at least Chatzidakis' Proposition~1.9 from~\cite{chatzidakis06}, which we now state in a slightly different form.
 
\begin{lemma}[Chatzidakis ~\cite{chatzidakis06}]
\label{chatzidakisprop}
Suppose that $a,b$ are tuples, $A$ is 
algebraically closed, $b = \cb(a/Ab)$, $a = \cb(b/Aa)$ and $\acl(Aa)\cap \acl(Ab) = A$.
Then there are mutually orthogonal nonmodular $U$-rank 
$1$-types 
$r_{1},..,r_{k}$, each nonorthogonal to $A$, and tuples $a_{1},..,a_{k},b_{1},..,b_{k}\notin A$ such that $\tp(a_{i}/A)$ and 
$\tp(b_{i}/A)$ are ${\bf R}_{i}$-analysable for each $i$ (where ${\bf R}_{i}$ is the set of $A$-conjugates of $r_{i}$), $a$ is interalgebraic with 
$(a_{1},..,a_{k})$ over $A$, $b$ is interalgebraic with $(b_{1},..,b_{k})$ over $A$, and for each $i$, $b_{i} = \cb(a_{i}/Ab)$ and $a_{i} = \cb(b_{i}/Aa)$.\qed
\end{lemma}

We now prove Theorem~\ref{ample-by-moishezon=moishezon} in the special case when the fibres of the generating family are of rank $1$.

\begin{proposition}
\label{theoremrank1}
Suppose that a stationary type $p(x)=\tp(a/A)$ has a nontrivial canonical generating family $q(x,y)=\tp(a,b/A)$ over $s(y)=\tp(b/A)$, with $q_b=\tp(a/Ab)$ of rank $1$.
\begin{itemize}
\item[(i)]
Then $q_b$ is nonorthogonal to $A$ and $p$ is almost analysable in $q_b$.
\item[(ii)]
If moreover $T$ has the CBP, then $p$ is almost internal to $q_b$.
\end{itemize}
\end{proposition}

Note that if the fibre $q_b$ is almost internal to a nonmodular minimal type $r$ then we can replace $q_b$ by $r$ in the conclusions~(i) and~(ii) above.
Hence Proposition~\ref{theoremrank1} does indeed imply a special case of Theorem~\ref{ample-by-moishezon=moishezon}.

\begin{proof}[Proof of~\ref{theoremrank1}]
Without loss of generality asume that $A=\acl(A)$.
Also note that $b$ and $\cb(a/Ab)$ are interalgebraic over $A$, and replacing $b$ by $\cb(a/Ab)$ preserves the hypotheses of the proposition.
So we may assume that $b=\cb(a/Ab)$.

Let $a'=\cb(b/Aa)$ and $b'=\cb(a'/Ab)$.
Then, as observed in the proof of Proposition~\ref{ucbp-zoe}, $b$ and $b'$ are interalgebraic over $A$.

We claim that {\em $a$ and $a'$ are interalgebraic over $Ab$}.
Note first that $b$ forks with $a'$ over $A$. For if
not, then $b$ is independent from $a$
over $A$, and so $b\in A$,
contradicting nontriviality.
So $a'\notin A$.
If $a'\in \acl(Ab)$, then, as $a'\in \acl(Aa)$, we
contradict the hypothesis that
$\acl(Aa)\cap \acl(Ab) = A$.
So $U(a'/Ab)\geq 1$.
As $U(a/Ab) = 1$
and $a'\in \acl(Aa)$ it follows that 
$a$ is interalgebraic with $a'$ over $Ab$ as required.

We have $b'=\cb(a'/Ab')$, $a'\in\cb(b'/Aa')$ and $\acl(Aa')\cap\acl(Ab')=A$.
So we can apply Lemma~\ref{chatzidakisprop} to $a',b',A$ to obtain $r_1,\dots,r_k$ and $a_1',\dots,a_k',b_1',\dots,b_k'$ satisfying the conclusion of that lemma (with the obvious notational changes).

\begin{claim}
$k=1$ and $r_1$ is nonorthogonal to $q_b$.
\end{claim}
\begin{proof}
Recall that $a_i'\notin A$ and $a'$ is interalgebraic with $(a_1'\dots,a_k')$ over $A$.
So $\acl(Aa_i')\cap\acl(Ab')=A$ and it follows that $a_i'\notin\acl(Ab')$.
Hence $U(a_i'/Ab')\geq 1$ for each $i=1,\dots,k$.
Since each $\stp(a_i'/Ab')$ is ${\bf R}_i$-analysable and the $r_i$'s are mutually orthogonal, it follows that $U(a'/Ab')\geq k$.
But as $b$ and $b'$ are interalgebraic over $A$ and $a$ and $a'$ are interalgebraic over $Ab$, it follows that $U(a'/Ab')=U(a/Ab)=1$.
So $k=1$ and $\stp(a'/Ab')$ is ${\bf R}_1$-analysable.
Hence $q_b=\tp(a/Ab)$ is ${\bf R}_1$-analysable, and so nonorthogonal to $r_1$ as desired.
\end{proof}

We can now prove (i).
Since by the claim $r_1$ is nonorthogonal to $q_b$, and since $r_1$ is nonorthogonal to $A$, it follows that $q_b$ is nonorthogonal to $A$.
Moreover, $b$ is interalgebraic with $b_1'$ over $A$.
It follows that $\tp(b/A)$ is almost analysable in $r_1$, and hence in $q_b$.
But $\tp(a/Ab)=q_b$.
So $\tp(a/A)$ is almost analysable in $q_b$.

For (ii) we already know from Corollary~\ref{corstrongcbp} that the CBP implies that $\tp(b/A)$ is almost $\mathbb{P}$-internal.
On the other hand, we have just seen that $\tp(b/A)$ is almost analysable in $q_b\in\mathbb{P}$.
Hence, $\tp(b/A)$ must be almost internal to $q_b$.
On the other hand, as the fibres have rank $1$, we know that $q$ almost separates points of $p$ (cf.~part~(b) of Remark~\ref{genrem}).
So $p^{\overline{M}}\subseteq\acl\big(As^{\overline{M}}\big)$, and hence $p$ is almost internal to $q_b$ as well.
This completes the proof of Proposition~\ref{theoremrank1}.
\end{proof}

The following lemma will enable us to reduce Theorem~\ref{ample-by-moishezon=moishezon} to Proposition~\ref{theoremrank1}. 

\begin{lemma}
\label{moishezon=veryample}
Suppose $r$ is a stationary type of rank at least $2$ that is almost internal to a nonmodular minimal type.
Then $r$ has a canonical generating family whose fibres are of rank $1$.
\end{lemma}

\begin{proof}
There is no harm in assuming $r$ to be over $\emptyset$.
It suffices to show
\begin{itemize}
\item[($*$)]
there is some rank $1$ stationary extension $q_b\in S(b)$ of $r$ such that there are realizations $c_{1}, c_{2}$ of $q_b$ which 
are independent over $\emptyset$. 
\end{itemize}
Indeed, assume such a $q_b$ exists and let $c\models q_b$.
Restricting $q_b$ we may assume $b = \cb(q_b)$.
Hence in order to see that $q:=\stp(c,b)$ is a canonical generating family for $r$ over $\stp(b)$ it remains to show that $\acl(c)\cap\acl(b)=\acl(\emptyset)$.
Suppose $d\in \acl(c)\cap \acl(b)$.
Then, as $q_b$ is stationary, if $c'\models q_b$ is independent from $c$ over $\emptyset$ then $d\in \acl(c)\cap \acl(c')=\acl(\emptyset)$, as desired.
So ($*$) suffices.

We next prove the lemma in the special case that $r = \tp(a_{1},..,a_{n})$ where $\{a_{1},..,a_{n}\}$
is an independent set of realizations of a nonmodular stationary rank $1$ type $p(x)$. By nonmodularity
of $p$ we can find, for each $i=2,..,n$, some $b_{i}$ such that $\tp(a_{1},a_{i}/b_{i})$ is stationary of
rank $1$, $b_{i}$ is the canonical base of $\tp(a_{1},a_{i}/b_{i})$, and $U(b_{i})\geq 2$. (See 
Chapter 2 of ~\cite{pillay96}.) Choose the 
$b_{i}$ as freely as possible, namely $\{b_{2},..,b_{n}\}$ is independent
 over $\emptyset$, and $\{a_{1},a_{i}\}$ is independent from all the other $a_{j}$'s together with all the 
$b_{j}$'s
 over $b_{i}$ (for each $i = 2,..,n$). Let $b = (b_{2},..,b_{n})$. Then $\stp(a_{1},..,a_{n}/b)$ has rank $1$
and $b$ is interalgebraic with the canonical base of the latter strong type. Let $(a_{1}',..,a_{n}')$ be a 
realization of
$\stp(a_{1},..,a_{n}/b)$, independent from $(a_{1},..,a_{n})$ over $b$.
Now if $U(a_1',a_i'/a_1,a_i) = 1$ then $(a_1',a_i')$ is independent from
$(a_1,a_i,b_i)$ over $(a_1,a_i)$.
But $b_i=\cb(a_1',a_i'/b)=\cb(a_1',a_i'/a_1,a_i,b)$, so $b_i\in\acl(a_1,a_i)$.
But then computing the rank of $\tp(a_1,a_i,b_i)$ shows that $U(b_i) \leq 1$, which is a contradiction.
So, for each $i$,  $(a_1,a_i)$ must be independent of $(a_1',a_i')$ over $\emptyset$.
It is now routine to conclude, using our free choice of the $b_{i}$'s, that $(a_{1}',...,a_{n}')$ 
is independent from $(a_{1},..,a_{n})$ over $\emptyset$. 
So ($*$) has been proved in the special case.

Now for the general case. Our assumption on $r$ implies that there is a realization $c$ of $r$, a model 
$M$ independent from $c$, a nonmodular rank $1$ type $p$ over $M$, and an $M$-independent set 
$\{a_{1},..,a_{n}\}$ ($n\geq 2$) of realizations of $p$ which is interalgebraic with $c$ over $M$. 
Denote $(a_{1},..,a_{n})$ by $a$. Let $b$ be given by the special case treated above, namely $\tp(a/M,b)$ has  
rank $1$, and if $a'$ realizes $\stp(a/M,b)$ independently from 
$a,b$ over $M$, then $a'$ is independent from $a$ over $M$. Note that $\tp(c/M,b)$ has rank $1$ too. 
Now let $(a,c)$ and $(a',c')$ be independent (over $M,b$) realizations of $\stp(a,c/M,b)$. As $a'$ is 
independent from $a$ over $M$ it follows (from the interalgebraicity of $c$ and $a$ over $M$), that $c'$
and $c$ are independent over $M$, and thus independent over $\emptyset$. So $c'$ and $c$ are  
realizations of $\stp(c/M,b)$ which are independent over $\emptyset$. This proves ($*$) for $r$.
\end{proof}

\subsection{Proof of Theorem~\ref{ample-by-moishezon=moishezon}}
Write $p=\tp(a/A)$, $q=\tp(a,b/A)$, and $s=\tp(b/A)$, such that $q$ generates $p$ over $s$ and $q_b$ is internal to a nonmodular minimal type $r$.

Without loss of generality assume that $A=\acl(A)$.
Let $b'=\cb(a/Ab)$.
Then $\tp(a/Ab')$ is stationary and almost internal to $r$, and $\acl(Aa)\cap\acl(Ab')=A$.
So the hypotheses of the theorem are preserved by replacing $b$ with $b'$.
We may therefore assume that $b=\cb(a/Ab)$ and $q$ is thus a canonical generating family.
If $q$ were a trivial generating family then $q_b=p$ and the theorem would follow.
So we may assume that $q$ is a nontrivial canonical generating family.

If $U(q_b)=1$ then the theorem follows from Proposition~\ref{theoremrank1}.
So we may assume that the fibres are of rank at least $2$.
Applying Lemma~\ref{moishezon=veryample} we obtain $c$ such that
\begin{itemize}
\item
$\tp(a/Abc)$ is stationary of rank $1$,
\item
$c\in\cb(a/Abc)$, and
\item
$\acl(Aba)\cap\acl(Abc)=\acl(Ab)$.
\end{itemize}
Letting $\hat c=\cb(a/Abc)$ we have that
\begin{itemize}
\item
$\hat c\notin A$,
\item
$\tp(a/A\hat c)$ is stationary of rank $1$,
\item
$\hat c=\cb(a/A\hat c)$, and
\item
$\acl(Aa)\cap \acl(A\hat c)=A$.
\end{itemize}
That is, $\tp(a,\hat c/A)$ is a nontrivial canonical generating family for $p$ over $\tp(\hat c/A)$ with fibre $\tp(a/A\hat c)$ of rank $1$.
Hence Proposition~\ref{theoremrank1} yields
\begin{itemize}
\item[(i)$^\prime$]
$\tp(a/A\hat c)$ is nonorthogonal to $A$ and $p$ is almost analysable in $\tp(a/A\hat c)$
\item[(ii)$^\prime$]
If moreover $T$ has the CBP then $p$ is almost internal to $\tp(a/A\hat c)$.
\end{itemize}
But since $q_b=\tp(a/Ab)$ is almost internal to $r$ so is the extension $\tp(a/Abc)$ and hence also $\tp(a/A\hat c)$.
So (i)$^\prime$ and (ii)$^\prime$ yield
\begin{itemize}
\item[(i)]
$r$ is nonorthogonal to $A$ and $p$ is almost analysable in $r$
\item[(ii)]
If moreover $T$ has the CBP then $p$ is almost internal to $r$.
\end{itemize}
This completes the proof of Theorem~\ref{ample-by-moishezon=moishezon}.
\qed


\begin{thebibliography}{10}

\bibitem{barlet}
D.~Barlet.
\newblock Espace analytique r\'eduit des cycles analytiques complexes compacts
  d'un espace anlytique complexe de dimension finie.
\newblock In {\em Fonctions de plusieurs variables complexes, II (S\'em. Fran\c
  cois Norguet, 1974--1975)}, volume 482 of {\em Lecture Notes in Math.}, pages
  1--158. Springer, 1975.

\bibitem{barlet79}
D.~Barlet.
\newblock Majoration du volume des fibres g\'en\'eriques et forme g\'eometrique
  du th\'eor\`me d'aplatissement.
\newblock {\em C. R. Acad. Sci. Paris. S\'er. A et B}, 288(1):A29--A31, 1979.

\bibitem{campanathesis}
F.~Campana.
\newblock {\em Application de l'espaces des cycles \`a la classification
  bimeromorphe des espaces analytiques {K}\"ahleriens compacts}.
\newblock PhD thesis, Universit\'e Nancy 1.
\newblock Pr\'epublication de l'universit\'e Nancy 1, no. 2, Mai 1980, p.
  1--163.

\bibitem{campana80}
F.~Campana.
\newblock Alg\'ebricit\'e et compacit\'e dans l'espace des cycles d'un espace
  analytique complexe.
\newblock {\em Mathematische Annalen}, 251(1):7--18, 1980.

\bibitem{campana81}
F.~Campana.
\newblock Cor\'eduction alg\'ebrique d'un espace analytique faiblement
  {K}\"ahl\'erien compact.
\newblock {\em Inventiones Mathematicae}, (63):187--223, 1981.

\bibitem{campanapeternell94}
F.~Campana and T.~Peternell.
\newblock Cycle spaces.
\newblock In H.~Grauert, T.~Peternell, and R.~Remmert, editors, {\em Several
  Complex Variables {VII}}, volume~74 of {\em Encyclopedia of Mathematical
  Sciences}, pages 319--349. Springer-Verlag, Berlin, 1994.

\bibitem{chatzidakis06}
Zo\'e Chatzidakis.
\newblock A note on canonical bases and one-based types in supersimple
  theories.
\newblock Preprint, 2006.

\bibitem{fujiki78}
A.~Fujiki.
\newblock Closedness of the {D}ouady spaces of compact {K}\"ahler spaces.
\newblock {\em Publication of the Research Institute for Mathematical
  Sciences}, 14(1):1--52, 1978.

\bibitem{fujiki82}
A.~Fujiki.
\newblock On the {D}ouady space of a compact complex space in the category
  {$\mathcal{C}$}.
\newblock {\em Nagoya Mathematical Journal}, 85:189--211, 1982.

\bibitem{Lascar}
D.~Lascar.
\newblock Sous groupes d'automorphismes d'une structure satur\'ee.
\newblock In {\em Logic colloquium '82 (Florence, 1982)}, volume 112 of {\em
  tud. Logic Found. Math.}, pages 123--134. North-Holland, 1984.

\bibitem{ret}
R.~Moosa.
\newblock A nonstandard {R}iemann existence theorem.
\newblock {\em Transactions of the American Mathematical Society},
  356(5):1781--1797, 2004.

\bibitem{moosa-ccs}
R.~Moosa.
\newblock The model theory of compact complex spaces.
\newblock In {\em Logic Colloquium '01}, volume~20 of {\em Lect. Notes Log.},
  pages 317--349. Assoc. Symbol. Logic, 2005.

\bibitem{sat}
R.~Moosa.
\newblock On saturation and the model theory of compact {K}\"ahler manifolds.
\newblock {\em Journal f\"ur die reine und angewandte Mathematik}, 586:1--20,
  2005.

\bibitem{pillay96}
A.~Pillay.
\newblock {\em Geometric Stability Theory}, volume~32 of {\em Oxford Logic
  Guides}.
\newblock Oxford Science Publications, Oxford, 1996.

\bibitem{pillay2000}
A.~Pillay.
\newblock Some model theory of compact complex spaces.
\newblock In {\em Workshop on {H}ilbert's tenth problem: relations with
  arithmetic and algebraic geometry}, volume 270, Ghent, 2000. Contemporary
  Mathematics.

\bibitem{pillay-cb}
A.~Pillay.
\newblock Notes on analysability and canonical bases.
\newblock E-print available at
  http://www.math.uiuc.edu/People/pillay/remark.zoe.pdf, 2001.

\bibitem{pillay01}
A.~Pillay.
\newblock Model-theoretic consequences of a theorem of {C}ampana and {F}ujiki.
\newblock {\em Fundamenta Mathematicae}, 174(2):187--192, 2002.

\end{thebibliography}

\end{document}